\documentclass{amsart}
\usepackage{amssymb}% some symbols, as \nrightarrow, for instance
\usepackage{mathrsfs}% to produce caligraphic text: The foliation \F(f), for instance
\usepackage{hyperref}
\usepackage{verbatim}
\usepackage{graphicx}
\usepackage{subfigure}
\usepackage{psfrag}
\usepackage{enumitem}

\newtheorem{theorem}{Theorem}[section]
\newtheorem{lemma}[theorem]{Lemma}
\newtheorem{corollary}[theorem]{Corollary}

{\theoremstyle{definition}

\newtheorem{remark}{Remark}}

\newcommand{\C}{\mathbb{C}}
\newcommand{\K}{\mathbb{K}}

\usepackage{xcolor}
\newcommand{\todo}[1]{{ \color{magenta} #1 }}
\author[Z. Jelonek]{Zbigniew Jelonek}
\author[G. Menani]{Gustavo Menani}
\author[M. Michalska]{Maria Michalska}

\address{Z.Jelonek, Instytut Matematyczny, Polska Akademia Nauk, \'Sniadeckich 8, 00-656 Warszawa, Poland  
}

\email{najelone@cyf-kr.edu.pl}

\address{G. Menani, Instituto de Ci\^encias Matem\'aticas e de Computa\c{c}\~ao, Universidade de S\~ao Paulo, S\~ao Carlos, SP, Brazil.}

\email{gmenani@usp.br}

\address{M. Michalska, Instituto de Ci\^encias Matem\'aticas e de Computa\c{c}\~ao, Universidade de S\~ao Paulo, S\~ao Carlos, SP, Brazil.}
\email{maria.michalska@usp.br}

\title[On Buzzard's Theorem]{On Buzzard's Theorem}

\thanks{The first named author was partially supported by the grant of Narodowe Centrum Nauki, number 2024/55/B/ST1/01412. The second named author was supported by FAPESP under grant number 2025/24335-9. The third named author was supported by FAPESP under grant number 2024/04171-9. The research of two first named authors was partially funded by the S\~ao Paulo Research Foundation (FAPESP) under grant number 2025/06706-0.}

\begin{document}
	
\maketitle

\begin{abstract}
	Over any infinite field we prove existence of polynomial automorphism with prescribed differentials at points. More precisely, let $\K$ be an infinite field and $a_1,\ldots, a_k;$ $b_1,\ldots , b_k$ be two sequences of different points in~$\K^n$, $n\ge 2$. For any sequence $L_1,\ldots, L_k\in SL(n,\K)$ there exists a polynomial automorphism $\Phi: \K^n\to \K^n$  with 
	jacobian one such that $\Phi(a_i)=b_i$ and $d_{a_i} \Phi= L_i$ for $i=1,\ldots, k.$
\end{abstract}

\section{Introduction}
In this note we study polynomial automorphisms over infinite fields with values and differentials prescribed in points. 
Let us provide a concise background for this problem.  In~\cite{jel} the first named author proved the following theorem:

\begin{theorem}\label{jelonek}
%Let $\K$ be an infinite field ,
Let $n\ge 2$ and $a_1,\ldots, a_k; b_1,\ldots , b_k$ be two sequences of different points of $\C^n$.
There there is a polynomial automorphism $\Phi: \C^n\to \C^n$ such that $\Phi(a_i)=b_i.$
\end{theorem}

\begin{remark}
{\rm This theorem in \cite{jel} was formulated over $\C$, however the method works for every infinite field.}
\end{remark}

Theorem 
\ref{jelonek} was generalized  by Buzzard in \cite{buz} who proved:

\begin{theorem}\label{buzzard}
Let  $n\ge 2$ and $a_1,\ldots, a_k; b_1,\ldots , b_k$ be two sequences of different points from $\C^n.$ Let $L_1,\ldots, L_k\in SL(n,\C)$ be linear isomorphisms with 
jacobian one. 
Then there is a polynomial automorphism $\Phi: \C^n\to \C^n$ with jacobian one such that $\Phi(a_i)=b_i$ and $d_{a_i} \Phi= L_i$ for $i=1,\ldots, k.$
\end{theorem}

Buzzard's theorem was later extended by Forstneric,  who proved in \cite{for} that over~$\C$ we can control not only $1$-jets, but also $m$-jets. Finally, the 
result of Forstneric was generalized in \cite{kal} by Arzhantsev,   Flenner,  Kaliman,  Kutzschebauch, and Zaidenberg, where it is shown that the claim holds not only on 
$\C^n$ but also on every flexible variety (see \cite{kal} for definition of flexibility).

However, all these generalizations hold only over $\C$ and their proofs are based mainly on analytic methods. The aim of this note is to give a proof of the Buzzard theorem over every
infinite field of any characteristic: \newpage

\begin{theorem}\label{glowne}
Let $\K$ be an infinite field. Let  $n\ge 2$ and $a_1,\ldots, a_k; $ $b_1,\ldots , b_k$ be two sequences of different points from $\K^n.$ Let $L_1,\ldots, L_k\in SL(n,\K)$ be linear isomorphisms with 
jacobian one. 
Then there is a polynomial automorphism $\Phi: \K^n\to \K^n$ with jacobian one, such that $\Phi(a_i)=b_i$ and $d_{a_i} \Phi= L_i$ for $i=1,\ldots, k.$
\end{theorem}

%We also estimate the degree of a polynomial $\Phi.$

\section{Proof of the main result}

Let us recall the general form of the Chinese Remainder Theorem which can be found for instance in \cite[Section 5]{ksiazka}:

\begin{theorem}\label{hermite}
Let $\K$ be a field of any characteristic. Let $x_1, x_2, \dots, x_k \in \K$ be pairwise distinct points, and $m_1, m_2, \dots, m_k$ be positive integers representing the node multiplicities. For each node $x_{i}$, let   ${g_i} \in \K[x]$ be a given remainder polynomial with $\deg(g_i) < m_i$. 

The system of polynomial congruences in the ring $\K[x]$:

$\begin{aligned}f(x)&\equiv g_{1}(x)\mathinner{\;\left(\mod \,(x-x_{1})^{m_{1}}\right)}\\ f(x)&\equiv g_{2}(x)\mathinner{\;\left(\mod \,(x-x_{2})^{m_{2}}\right)}\\ &\;\;\vdots \\ f(x)&\equiv g_{k}(x)\mathinner{\;\left(\mod \,(x-x_{k})^{m_{k}}\right)}\end{aligned}$

has a unique solution $f(x) \in \K[x]$ satisfying the degree constraint  $\deg (f)<\sum _{i=1}^{k}m_{i}$.
\end{theorem}

\begin{corollary}\label{glowne}
Let $\K$ be a field of any characteristic. Let $x_1, x_2, \dots, x_k \in \K$ be pairwise distinct points and $a_1, a_2, \dots, a_k;$ $b_1,\dots, b_k$ be elements from $\K.$ 
Then there exist  polynomials $f, p \in \K[x]$ such that 
$$p(x_i)= a_i$$
and
$$f(x_i)=a_i, f'(x_i)=b_i$$
for $ i=1,\ldots, k$. Moreover, $\deg p <k$ and  $\deg f<2k$. %Similarly, there is a plynomial $G$ of degree less than $k$, such that $G(x_i)=a_i$ for $i=1,\ldots,k.$
\end{corollary}

\begin{proof}
In Theorem \ref{hermite} set  $g_i(x)=a_i+b_i(x-x_i)$ for $ i=1,\ldots, k$. If we also set $m_i=2$ for $ i=1,\ldots, k$, there exists a polynomial $f$ of degree less than $2k$ such that
$f(x)\equiv g_{i}(x) \mod \,(x-x_{i})^{2}$, i.e., $$f(x)=a_i+b_i(x-x_i)+(x-x_i)^2q(x)$$ 
for some polynomial $q\in \K[x]$.  This gives at once $f(x_i)=a_i.$
Moreover, $f'(x)=b_i+(x-x_i)^2q'(x)+2(x-x_i)q(x)$, i.e., $f'(x_i)=b_i.$ 

To obtain the polynomial $p$ consider $m_i=1$ for $ i=1,\ldots, k$.
\end{proof}

Now we will give here the simple proof 
 of Theorem \ref{jelonek} that shows the result and proof of \cite{jel} is valid in fact over any infinite field. Let us state it in a slightly more general version:

\begin{theorem}\label{thm:pointautomorphism}
Let $\K$ be an infinite field. Let $n\ge 2$. 
For any two sets   $\{a^1,\ldots,a^k\}$ and $\{b^1,\ldots,b^k\}$  of $k$ distinct points in~$\K^n$ there exists a polynomial automorphism $F:\K^n \to \K^n$ of degree at most $k^2$ and jacobian determinant $1$ such that
$$ F(a^i) = b^i $$
for each $i=1,\dots,k$.
\end{theorem}

\begin{proof}

It is easy to see that there exist linear isomorphisms $A,B\in SL_n(\K)$ such that the points $A(a^1),\dots, A(a^k)$  are in a general position such that no two points have the same coordinate, i.e.  $a_i^j\neq a_i^l$ for all $i\leq n$, $j, l\leq k$ and $j\neq l$, and the points $B(b^1),\dots,B(b^k)$ are such that no two first coordinates are the same. If there exists an automorphism $F$ satisfying the claim in these new coordinates, then $B^{-1}\circ F\circ A$ satisfies the claim in old coordinates. Thus without loss of generality we may assume that the sets  $\{a^1,\ldots,a^k\}$ and $\{b^1,\ldots,b^k\}$ are in a general position.

By Corollary \ref{glowne} for $1\leq j \leq n-1$ exist polynomials $p_j\in \K[t]$ of degree at most~$k$  such that 
$$p_j(a^{i}_{j+1})=b^{i}_{j}-a^{i}_{j}$$
for every $i=1,\dots, k$.
Let
$$P(x_1,\dots, x_n):=(x_1+p_{1}(x_{2}),\dots,\ x_{n-1}+p_{n-1}(x_{n}),\ x_n).$$
Then $P$ is a tame automorphism and $P(a^i)=(b^{i}_{1},\dots, b^{i}_{n-1}, a^{i}_{n}).$

Let $p_n\in\K[t]$ be such that 
$$p_n(b^{i}_{1})=b^{i}_{n}-a^{i}_{n}$$
for every $i=1,\dots,k$. 
Take
$$H(x_1,\dots,x_n):=(x_1,\dots,x_{n-1},\ x_n+p_n(x_1)).$$
Then  $H$ is a tame polynomial automorphisms of $\K^n.$ Let $$F=  H\circ P.$$ 
We have
$$F(a^i)=  (H\circ P)(a^i)=H(b^{i}_{1},\dots, b^{i}_{n-1}, a^{i}_{n})=(b^{i}_{1},\dots, b^{i}_{n})=b_i $$
%Thus $H\circ P(a_i)=b_i$
 for $i=1,\dots, k.$ Since $\deg P, \deg H\le k$ we have $\deg F\le k^2.$
\end{proof}

To complete the proof of the main result we need now the following: 

\begin{lemma}\label{lem:differentials}
	Under notation of Theorem~\ref{thm:main} there exists an automorphism $G$ of $\K^n$ of degree at most $(2k)^{kn^2} $ that fixes the set  
	$\{b^1,\ldots,b^k\}$ and satisfies $d_{b^j}G=L_j$ for $j=1,\dots, k$.
\end{lemma}

\begin{proof}
We can assume that points  $\{b_1,\ldots,b_k\}$ are in a general position.
Denote by $E_{rs}\in M_n(\K)$ the  matrix with only the entry $(r,s)$ equal to $1$ and $0$ otherwise. Then we write an elementary matrix in the form $\textup{Id}+c E_{rs}$ with $r\neq s$ and $c\in\K$. Gaussian elimination over $\K$ allows us  to write each $L_i$ as a product of at most $n^2$ elementary matrices. Therefore, allowing elementary matrices to be identity, we can write
$$L_i=M_i^1\cdots M_i^{m}, $$
where
$$M_i^j= \textup{Id}+c_i^j E_{r_js_j}$$
with  $c_i^j\in\K$, and $r_j\neq s_j$  and $m\leq k n^2$ is independent on $i$ and for fixed $j$ all matrices $M_i^j$ are of the same type.
For $1\leq j\leq m$ define a tame automorphism $g_j$ of $\K^n$ as
$$g_j(x_1,\dots,x_n)=(x_1,\dots,x_{r_j}+f_j(x_{s_j}),\dots,x_n)$$
with $f_j\in\K[t]$ of degree at most $2k$
by Corollary \ref{glowne}  and 
satisfying
$$  f_j(b^i_{s_j})=0 \quad \text{and} \quad  f_j'(b^i_{s_j})=c_i^j .$$

Note that
$$g_j(b^i)=b^i \quad \text{and} \quad d_{b^i}g_j=M_i^j.$$

Let
$$G: =g_1\circ\cdots\circ g_m.$$
Then
$$G(b^i)=b^i \quad \text{and} \quad d_{b^i}G=L_i.$$
\end{proof}

Now we can prove our main result:

\begin{theorem}\label{thm:main}
Let  $\K$ be an infinite field and $n\ge 2.$ Let  $\{a^1,\ldots,a^k\}$ and $\{b^1,\ldots,b^k\}$ be sets of $k$ distinct points in $\K^n$,  and $L_1,\ldots,L_k\in SL(n,\K)$. 
There exists a polynomial automorphism $\Phi$ of $\K^n$ of degree at most $k^2(2k)^{kn^2}$
such that
$$ \Phi(a^i)=b^i \quad \text{and}\quad   d_{a^i}\Phi=L_i  $$
for each $i=1,\dots,k$.
\end{theorem}

\begin{proof}
By Theorem~\ref{thm:pointautomorphism} there exists a polynomial automorphism $F$  such that
$$F(a^i)=b^i$$
for every $i=1,\dots,k$. 
For each $i=1,\dots,k$ set
$$K_i=L_i\circ(d_{a^i}F)^{-1}\in SL_n(\K).$$

By Lemma~\ref{lem:differentials} there is  an automorphism $G$ that fixes the points $b_i$ and satisfies the condition  $d_{b^j}G=K_j$ for $j=1,\dots, k.$
Now it is enough to put 
$$\Phi:= G\circ F. $$
\end{proof}

\begin{corollary}
	Under notation of Theorem~\ref{thm:main}
	%The polynomial automorphism  of Theorem~\ref{thm:main} is homotopic to identity. More precisely, 
	there exists a polynomial isotopy $\Psi: \K^n\times \K\to \K^n$  such that $\Psi(\cdot, t)$ for every $t\in \K$ is an automorphism with jacobian determinant $1$,
	$$ \Psi(\cdot,0) = \textup{id}_{\K^n}$$ 
	%\quad \text{and} \quad
	and
	$$ \Psi(\cdot,1)= : \Phi $$
	satisfies the claim of Theorem~\ref{thm:main}.
\end{corollary}

\begin{proof}
In the construction of Theorem~\ref{thm:pointautomorphism} it suffices to consider 
$$ P_t(x):= (x_1+tp_{1}(x_{2}),\dots,\ x_{n-1}+tp_{n-1}(x_{n}),\ x_n) $$
and 
$$H_t(x):=(x_1,\dots, x_{n-1},\ x_n+tp_n(x_1)).$$
Thus $F_t:= H_t\circ P_t$ gives homotopy from  the identity to the mapping  $F.$ 

In the construction of Theorem~\ref{thm:main} similarly one considers
$$(g_j)_t(x_1,\dots, x_n)=(x_1,\dots,x_{r_j}+tf_j(x_{s_j}),\dots,x_n)$$
and
$$G_t: =(g_1)_t\circ\cdots\circ (g_m)_t.$$

\end{proof}

\begin{corollary}\label{glowne1}
Let $\K$ be an algebraically closed field of any characteristic. Let~$X$ be a smooth affine variety over $\K$ of dimension $n$ and let $a_1,\dots,a_k$ be different points of $X$.
Take different points $b_1,\dots, b_k$ in $\K^{2n+1}$ and  linear $n$ dimensional spaces $H_1,\dots, H_k$ in $\K^{2n+1}$.
Then there is a closed embedding $\phi: X\to \K^{2n+1}$, such that $\phi(a_i)=b_i$ and $T_{b_i} \phi(X)=H_i$ for $i=1,\dots,k.$
\end{corollary}

\begin{proof}
It is well known that there is a closed embedding $\iota: X\to \K^{2n+1}$ (see for instance \cite{jel}). Take $X'=\iota(X)$ and $\iota(a_i)=a'_i$ for $ i=1,\dots, k.$ Let $L_i\in SL(n,\K)$ be linear isomorphism, such that
$L_i(T_{a'_i} X')=H_i,\ i=1,\dots, k.$ Now let $\Phi:\K^{2n+1}\to\K^{2n+1}$ be a polynomial automorphism with jacobian one, such that $\Phi(a'_i)=b_i$ and $d_{a'_i}\Phi=L_i$
for $i=1,\dots, k.$ Thus $\phi=\Phi\circ\iota$ is a desired embedding.
\end{proof}

\begin{remark}	
{\rm Theorem~\ref{thm:main}  holds also over finite fields provided the number $k$ of points is small enough in comparison to the order of the field.}
\end{remark}

\end{document}